\documentclass{amsart}

\usepackage{comment}
\usepackage{url}
\usepackage{graphicx}
\usepackage{verbatim}
\usepackage{tensor,color}
\newtheorem{thm}{Theorem}
\newtheorem{cor}{Corollary}
\newtheorem{lem}[cor]{Lemma}
\newtheorem{prop}[cor]{Proposition}

\theoremstyle{definition}

\theoremstyle{remark}
\newtheorem{rem}{Remark}

\numberwithin{equation}{section}

\renewcommand{\t}{t}

\newcommand{\set}[1]{\left\{#1\right\}}
\newcommand{\Real}{\mathbb R}

\newcommand{\func}[1]{\ensuremath{\operatorname{#1} } }

\newcommand{\Div}[0]{\func{div}}

\newcommand{\dist}[0]{\mathrm{dist}}

\newcommand{\tr}[0]{\func{tr}}

\newcommand{\R}{\mathbb{R}}

\newcommand{\Jdu}[2]{J_{#1}{}^{ #2}}

\newcommand{\Ric}[0]{\func{Ric}}
\renewcommand{\d}{\mathrm{d}}

\newcommand{\vol}{\mu_g}



\date{March 27, 2013}

\address{Department of Pure Mathematics and Mathematical Statistics, University of Cambridge, Cambridge, UK.}
\email{j.bernstein@dpmms.cam.ac.uk}
\address{Department of Applied Mathematics and Theoretical Physics, University of Cambridge, Cambridge, UK.}
\email{t.mettler@damtp.cam.ac.uk}

\title[2D gradient Ricci solitons]{Two-dimensional gradient Ricci solitons revisited}
\author[J. Bernstein \and T. Mettler]{Jacob Bernstein \and Thomas Mettler}
\thanks{The first author was supported by the EPSRC Programme Grant entitled ``Singularities of Geometric Partial Differential Equations'' grant number EP/K00865X/1. The second author was supported by  Forschungsinstitut f\"ur Mathematik at ETH Z\"urich and by a postdoctoral fellowship of Schweizerischer Nationalfonds SNF, PA00P2\textunderscore 142053.}

\begin{document}

\begin{abstract}
In this note, we complete the classification of the geometry of non-compact two-dimensional gradient Ricci solitons. As a consequence, we obtain two corollaries:  First, a complete two-dimensional gradient Ricci soliton has bounded curvature. Second, we give examples of complete two-dimensional expanding Ricci solitons with negative curvature that are topologically disks and are not hyperbolic space. 
\end{abstract}

\maketitle

\section{Introduction}
Recall, a triple $(M,g,X)$ consisting of a smooth manifold $M$, a smooth Riemannian metric $g$ on $M$ and a smooth vector field $X \in \Gamma(TM)$ is said to be a \textit{Ricci soliton with expansion constant} $\lambda$, if it satisfies
\begin{equation}\label{ricsol}
-2\Ric(g)=L_X g-2\lambda g
\end{equation}
where $L_Xg$ denotes the Lie derivative of $g$ with respect to $X$. For such a triple the Ricci flow equation
$$
\frac{\d}{\d t}g_t=-2\Ric(g_t)
$$
with initial condition $g_0=g$ has, in the $t$-interval on which $1-2\lambda t>0$, the solution
$$
g_t=(1-2\lambda t)\phi_t^*g,
$$
where $\phi_t : M \to M$ is the (possibly locally defined) time $t$ flow of $X$. When $\lambda=0$, the soliton is said to be steady, when $\lambda>0$, the soliton is said to be shrinking, and when $\lambda<0$, the soliton is said to be expanding. 

Let $\nabla$ denote the Levi-Civita connection of $g$. If $X=\nabla f$ for some smooth real-valued function $f$ on $M$, then we say $(M,g,f)$ is a \emph{gradient Ricci soliton} and $f$ is its \emph{soliton potential}.

It has been known for some time that closed two-dimensional Ricci solitons must have constant curvature -- see \cite[Theorem 10.1]{MR954419} and \cite{MR1375255} -- also \cite{MR2231924} and \cite{MR1094470}.  It was also well-known that any two-dimensional gradient Ricci soliton must admit a non-trivial Killing vector field -- see, for instance, the editor's footnote of \cite[p.~241-242]{MR2145154}.  However, the possible geometries in the open case were not completely determined. 

The most complete analysis in this direction was in \cite[Appendix A and B]{MR1357953} where all rotationally symmetric two-dimensional self-similar solutions to the logarithmic fast diffusion equation were explicitly computed.  
As the logarithmic fast diffusion equation is the evolution of the conformal factor -- with respect to a fixed flat background -- of a metric evolving under the Ricci flow, the analysis of \cite{MR1357953} gives many examples of two-dimensional Ricci solitons.   More in the spirit of the present note is  \cite{Kotschwar}, where the Killing vector field is used to show that all positively curved  two-dimensional expanding Ricci solitons are the rotationally symmetric examples of \cite{MR1969853}.  A more complete overview of known results may be found in \cite[Chap. 1]{MR2302600}. 

Motivated by an observation of \cite{JacobThomasEntropy} connecting two-dimensional steady gradient Ricci solitons with minimal surface metrics, we reduce the classification of two-dimensional gradient Ricci solitons to the analysis of a simple autonomous first order \textsc{ode}. The analysis of this \textsc{ode} gives:
 
\begin{thm}\label{ListofSolitonsThm}
 The following twelve one-parameter families of rotationally symmetric metrics are gradient Ricci soliton metrics on either the disks  $\mathbb{D}=\set{x:|x|<1}\subset \Real^2$ or $\Real^2$ or the annuli $\mathbb{D}_*=\mathbb{D}\backslash \set{0}$ or $\Real^2_*=\Real^2\backslash \set{0}$:
 \begin{enumerate}
 \item Steady solitons:
   \begin{enumerate} 
   \item $(\Real^2, g_1(\nu))$ -- The cigar solitons of Hamilton \cite{MR954419} -- i.e. complete positively curved steady solitons  -- the parameter $\nu$ corresponds to scaling of the metric;
   \item  $(\Real^2, g_2(\nu))$ -- Incomplete steady solitons with unbounded negative curvature -- called the exploding solitons in \cite{MR2302600} -- the parameter $\nu$ corresponds to scaling of the metric;
   \item  $(\Real^2_*, g_3(\nu))$ -- Incomplete negatively curved steady solitons asymptotic to a cylinder of radius $\nu\geq 0$ at $0$ and to $g_2(1)$ at $\infty$.
   \end{enumerate}
  \item Shrinking solitons:
   \begin{enumerate}
    \item $(\mathbb{D}, g_{4}^\pm(\nu))$ -- Two families of incomplete positively curved shrinking solitons which extend smooth\-ly to $\bar{\mathbb{D}}$ so that $\partial \bar{\mathbb{D}}$ is totally geodesic and has length $2\pi$ -- the parameter $\nu=\dist(0, \partial \mathbb{D})$ satisfies $\nu\in (1,\infty)$ for $g_4^+$ and $\nu\in (\frac{\pi}{2}, \infty)$ for $g_4^-$;     \item $(\Real^2, g_{5}(\nu))$ -- Incomplete negatively curved  shrinking solitons asymptotic to  $g_2(\nu)$;
  \end{enumerate}
   \item Expanding solitons:
     \begin{enumerate}
       \item $(\Real^2, g_{6}(\nu))$ -- Complete positively curved expanding solitons asymptotic to the end of a flat cone of angle $\nu\in(0, 2\pi)$ -- see \cite{MR1969853};
     \item $(\Real^2, g_{7}(\nu))$ -- Complete negatively curved expanding solitons asymptotic to to the end of a flat cone of angle $\nu\in (2\pi,\infty)$;
       \item $(\Real^2_*, g_{8}(\nu))$ -- Complete negatively curved expanding solitons asymptotic to a hyperbolic cusp at $0$ and to the end of a flat cone of angle $\nu\in(0,\infty) $ at $\infty$;
       \item $(\mathbb{D}_*, g_{9}(\nu))$  -- Incomplete negatively curved expanding solitons asymptotic to to the end of a flat cone of  angle $\nu \in(0,\infty)$ at $0$ and extending  smoothly to $\bar{\mathbb{D}}_*$ so that $\partial \bar{\mathbb{D}}$ is totally geodesic and of length $2\pi$;
       \item $(\Real^2, g_{10}(\nu))$ -- Incomplete negatively curved  expanding solitons asymptotic to  $g_2(\nu)$;
       \item $(\Real^2_*, g_{11}(\nu))$ -- Incomplete negatively curved expanding solitons asymptotic to a hyperbolic cusp at $0$ and to $g_2(\nu)$ at $\infty$.  
        \item $(\mathbb{D}_*, g_{12}(\nu))$ -- Incomplete negatively curved  expanding solitons asymptotic to  $g_2(\nu)$ at $0$ and extending  smoothly to $\bar{\mathbb{D}}_*$ so $\partial \bar{\mathbb{D}}$ is totally geodesic and of length $2\pi$;
       \end{enumerate}
\end{enumerate}
\end{thm}
\begin{rem}
 The metrics $g_4(\nu), g_9(\nu)$ and $g_{12}(\nu)$ may be ``doubled'' to give $C^{2,1}$ metrics on a doubled surface. The doubled metrics are not themselves Ricci solitons.
\end{rem}
\begin{rem}
All these metrics seem to have been written down in \cite{MR1357953}, however their geometric properties are not fully discussed there.
\end{rem}

A further consequence of the analysis of the \textsc{ode} is that the above list contains all possible models of non-constant curvature two-dimensional gradient Ricci solitons.
\begin{thm}\label{MainClassificationThm}
 Let $(M, g, f)$ be a smooth gradient Ricci soliton with $M$ a connected orientable two-manifold. If $M_*=\set{p\in M: \nabla_g K_g\neq 0}$ is non-empty and $\pi_1(M_*)$ is cyclic, then  
there is a $0<R\leq \infty$, a  metric $\bar{g}$ on $\mathbb{D}_*(R)=\set{x:0<|x|<R}\subset \Real^2$ and an isometric immersion $i:M_*\to \mathbb{D}_*(R)$ so that:
 \begin{enumerate}
   \item $K_{\bar{g}}\neq 0$; 
  \item $(\mathbb{D}_*(R), \bar{g})$ is maximal in the sense that if  $(\hat{M}, \hat{g}, \bar{f})$ is a two-dimensional connected gradient Ricci soliton  and $\hat{i}:(\mathbb{D}_*(R), \bar{g})\to (\hat{M},\hat{g})$ is a smooth isometric embedding, then $\hat{M}\backslash \hat{i}(\mathbb{D}_*(R))=\set{p}$ and $\nabla_{\hat{g}} K_{\hat{g}}(p)=0$;
  \item $\bar{g}$ is rotationally symmetric, that is in polar coordinates $(r,\theta)$
   $$\bar{g}= \d r^2+b^2(r) \d\theta^2$$
    for some positive functions $b\in C^\infty((0,R))$;
   \item There are $\alpha,\beta>0$ so that if $\bar{g}_\alpha= \d r^2+\alpha^2 b^2(r) \d\theta^2$, then $(\mathbb{D}_*(R),\beta^2 \bar{g}_\alpha)$
        may be isometrically embedded in one of the models given in Theorem \ref{ListofSolitonsThm}.
\end{enumerate}
\end{thm}
\begin{rem}
If $g= \d r^2+b^2(r) \d\theta^2$ is a metric on $\mathbb{D}_*(R)$, then for any $\alpha>0$ the metric $g_{\alpha}=\d r^2+\alpha^2 b^2(r) \d\theta^2$ on $\mathbb{D}_*(R)$ is locally isometric to $g$. In particular, $g$ is a gradient Ricci soliton metric if and only if $g_\alpha$ is.  Likewise scaling a Ricci soliton metric gives a new Ricci soliton metric with scaled expansion constant.
\end{rem}
\begin{rem}
Neither the orientability condition on $M$ nor the condition on $\pi_1(M_*)$ can be removed because several locally isometric but not globally isometric subsets of a fixed model soliton may be glued together to yield counter-examples.  
\end{rem}

Theorems \ref{ListofSolitonsThm} and \ref{MainClassificationThm} resolve two questions regarding the structure of non-compact two-di\-men\-sional solitons (see \cite[p.~61]{MR2302600}).
Specifically,
\begin{cor}
 A complete two-dimensional gradient Ricci soliton has bounded curvature.
\end{cor}
\begin{cor}
 There exist complete two-dimensional expanding Ricci solitons with negative curvature that are topologically disks and are not hyperbolic space -- namely the metrics $g_7$ and the universal covers of the metrics $g_8$.
\end{cor}
\begin{rem}
 The metrics $g_7$ are parabolic and appear in \cite{MR1357953} -- though this does not seem to be widely known.  The universal covers of the metrics $g_8$ are non-parabolic.
\end{rem}

Finally, in Section \ref{VarSec} we give a simple variational characterization of two-dimensio\-nal gradient Ricci solitons and observe a connection with \cite{JacobThomasEntropy}.
 
\subsection*{Acknowledgments}
The authors would like to thank B. Kotschwar for his comments and for pointing out a serious error in an early draft.  We also thank P. Topping for bringing \cite{MR1357953} to our attention.

\section{Properties of two-dimensional Ricci solitons}\label{Sec2}
We will henceforth assume $M$ to be a connected  two-manifold. Then for a gradient Ricci soliton equation \eqref{ricsol} becomes
\begin{equation*}
 \nabla^2 f+ K g -\lambda g=0
\end{equation*}
which is equivalent to
$$
\Delta f =2(\lambda-K) \quad \text{and}\quad
\mathring{\nabla}^2 f=0.
$$ 
Here $K$ denotes the Gauss curvature, $\Delta$ the Laplacian, and $\mathring{\nabla}^2$ the trace-free Hessian of $g$. Following,  \cite[Appendix A]{JacobThomasEntropy}, covariant differentiation of $\mathring{\nabla}^2 f=0$ together with straightforward manipulations gives
$$
0=K \d f+\frac{1}{2} \d \Delta f=K \d f-\d K. 
$$
At a point $p \in M$ where $K\neq 0$ we thus have
\begin{equation*}
\d f = \d \log |K|
\end{equation*}
which yields
$$
\nabla^2 \log | K|= \nabla^2 f= (\lambda- K)g. 
$$
From this we see that
\begin{equation} \label{MasterEqn}
 \Delta \log | K|  =2(\lambda-K) \quad \text{and} \quad \mathring{\nabla}^2 \log | K| =0. 
\end{equation}
We have shown:
\begin{prop}\label{charsol}
Let $(M,g)$ be a  Riemannian two-manifold with non-zero Gauss curvature. Then $g$ is a gradient Ricci soliton metric with expansion constant $\lambda$ if and only if
$$
\mathring{\nabla}^2 \log |K| =0\quad \text{and}\quad \Delta \log |K| =2(\lambda-K).
$$
Moreover, if $(M, g, f)$ is a gradient Ricci soliton, then we may take $f=\log |K|$.  
\end{prop}
\begin{rem}\label{killing}
Recall (c.f.~\cite[p.~11]{MR2302600}) that if $(M,g,\nabla f)$ is a gradient Ricci soliton and $M$ is an oriented surface, then $J(\nabla f)$ is a Killing vector field for $g$. Here, as usual, $J$ denotes the integrable almost complex structure induced by $g$ and the orientation, i.e.~$J$ rotates $v \in TM$ by $\pi/2$ in positive direction.  
\end{rem}
Proposition \ref{charsol} and Remark \ref{killing} reduce the classification of gradient Ricci solitons in dimension two to the study of a simple \textsc{ode}. 
Indeed, assume $(M,g)$ is a gradient Ricci soliton and $g$ does not have constant curvature.  We first observe that $M$ has non-vanishing curvature and has non-constant curvature almost everywhere.
\begin{prop} \label{densityprop}
If  $(M,g,f)$ is a connected two-dimensional gradient Ricci soliton and $M_*=\set{p\in M: \nabla K (p)\neq 0}\neq \emptyset$, then $K\neq 0$ on all of $M$ and $M\backslash M_*$ consists of isolated points.
\end{prop}
\begin{proof}
Set $M_*{}_*= \set{p\in M: K(p)\neq 0, \nabla K (p)\neq 0}$ and
decompose $B=M\backslash M_*{}_*$ into $B_1=\set{p\in M: K(p)=0, \nabla K (p)\neq 0}$, $B_2=\set{p\in M: K(p)=0, \nabla K (p)= 0}$ and  $B_3=\set{p\in M: K(p)\neq 0, \nabla K (p)= 0}$. As $M$ has non-constant curvature, $M_*{}_*$ is non-empty.  We claim it is dense in $M$.
This follows from standard unique continuation results for uniformly elliptic equations \cite{MR0092067} and the fact that $K$ satisfies
$$
\Delta K-\nabla f \cdot \nabla K +2K^2+\lambda K =0
$$
on $M$ -- see \cite[Lemma 1.11]{MR2302600}.
Hence, $B_1=\emptyset$.  Indeed, by Proposition \ref{charsol}, $\nabla f-\nabla \log |K_g|$ is covariant constant on $M_*{}_*$ and so is bounded on all of $M$.

Fix $p\in M$ and let $(r,\theta)$ be normal coordinates about $p$. In these coordinates, 
\begin{equation*}
 \Delta u= \partial_r^2 u +\frac{1}{r} \partial_r u +\frac{1}{r^2} \partial_\theta^2 u+ \alpha_1 \partial_r u +\frac{1}{r} \alpha_2 \partial_\theta u
\end{equation*}
where $\alpha_i$ are smooth functions -- depending on $u$ and $g$ -- defined near $p$. 
We next show that $B_2=\emptyset$.  Indeed,
by  \cite[Lemma 1.11]{MR2302600}, the unique continuation result \cite{MR0092067} and the expansion of the Laplacian in normal coordinates  near any $p\in B_2$ we have
\begin{equation*}
K=\alpha_1 r^n  \cos n\theta+ \alpha_2^2 r^n \beta \sin n \theta +o(r^n)
\end{equation*}
for some $n\geq 2$ and $\alpha^2_1+\alpha^2_2\neq 0$.  However, if this is true then there would exist a point in $B_1$ and so $B_2=\emptyset$.

Finally, we conclude that $B_3$ consists of isolated points.
Indeed, $\nabla \log |K_g|$ is a conformal Killing vector field on $M$ and so is (locally) the real part of a non-trivial holomorphic vector field.  In particular, it can only vanish at isolated points.
\end{proof}

By Proposition \ref{charsol}, in a neighborhood  of every point $p \in M_*$ there exists a non-trivial Killing vector field for $g$. It follows -- see  ~\cite[Lemma 1.18]{MR2302600} -- that there exists a local coordinate system $(r,\theta)$ defined on some neighborhood $U_p$ of $p$, such that
$$
g=\d r^2+b(r)^2\d \theta^2
$$
for some smooth positive real-valued function $b$ defined in some neighborhood of $r(p)$. On $U_p$ we thus have
$$
K(r)=-\frac{b^{\prime\prime}(r)}{b(r)}.
$$
Since $K(p)\neq 0$ we can assume (after possibly replacing $p$ with a nearby point), that $b^{\prime}(r)$ is non-vanishing in some simply connected neighborhood $V_p$ of $p$.  
Setting $\t=\frac{1}{4}b(r)^2$ it follows that $(\t,\theta) : V_p \to \R^2$ is a local coordinate system so that
\begin{equation}\label{speccoord}
g=\frac{a(\t)^2}{\t}\d \t^2+4\t \d \theta^2 
\end{equation}
for some positive function $a$ defined in some neighborhood of $\t(p)$. For $t>0$ in the domain of definition of $a$, this metric is smooth and has Gaussian curvature
\begin{equation}\label{curvform}
K(\t)=\frac{1}{2} \frac{a'(\t)}{a(\t)^3}. 
\end{equation}
For later reference we remark that the geodesic curvature of the curves $\set{t=t_0}$ is
\begin{equation} \label{geocurv}
 \kappa= \frac{1}{2 \sqrt{t_0 a(t_0)}}
\end{equation}
which follows easily from the first variation formula.

As $J(\nabla\log |K|)$ is a Killing vector field, we may take
\begin{equation}\label{ode1}
\frac{a(t)}{2\t} \frac{\t}{ a(t)^2}\left(\log \left|\frac{1}{2}\frac{a'(t)}{ a(t)^3} \right|\right)'=2\mu
\end{equation}
for some constant $\mu$ which is non-vanishing since $\d K(p)\neq 0$. 
On the other hand, from Proposition \ref{charsol} we obtain 
\begin{equation}\label{ode2}
 \frac{1}{2a(t)} \left( \frac{2\t}{a(t)} \left( \log \left|\frac{1}{2}\frac{a'(t)}{a(t)^3} \right|\right)'\right)' =- \frac{a(t)'}{a(t)^3} +2\lambda.
\end{equation}
Combining \eqref{ode1} and \eqref{ode2} gives the autonomous first-order \textsc{ODE}:
\begin{equation}\label{MasterODE}
a'(t)=4 \mu a(t)^2 \left(\frac{\lambda}{2\mu} a(t)-1\right). 
\end{equation}

Conversely we have:
\begin{lem}\label{smoothext}
If $a>0$ is defined on $(0,T)$ with $0<T\leq \infty$ and solves \eqref{MasterODE} for some $\lambda,\mu\neq 0$,  then every solution $b$ to 
\begin{equation}\label{intode}
a\left(\frac{b(r)^2}{4}\right)b^{\prime}(r)=1,
\end{equation}
defined on $(0,R)$ gives rise to a smooth rotationally symmetric gradient Ricci soliton 
$
g=\d r^2+b(r)^2 \d \theta^2
$
on $\mathbb{D}_*(R)$. Furthermore, there exists a solution $b$ to \eqref{intode} so that the associated metric extends smoothly to $\mathbb{D}(R)$ if and only if $\lim_{t \to 0}a(t)=1$.  In this case $ K(0)=\lambda-2\mu$ and $ \nabla K(0)=0$. 
\end{lem}
\begin{proof}
The first statement is an immediate consequence of the calculations above. To prove the second statement we first recall the standard result -- see~\cite[Lemma A.2]{MR2302600} -- that a metric of the form $g=\d r^2+b(r)^2 \d \theta^2$ for some smooth non-vanishing function $b$ on $(0,R)$ extends smoothly to $\mathbb{D}(R)$ if and only if 
\begin{equation}\label{smoothextcond}
\lim_{r \to 0}b^{\prime}(r)=\pm 1, \quad \lim_{r \to 0} \frac{\d^{2k}b(r)}{\d r^{2k}}=0 \;\;\forall k \in \mathbb{N}_0. 
\end{equation}
Suppose the metric associated to a solution of \eqref{intode} extends smoothly to the topological disk, then it follows with the continuity of the function $a$ and (\ref{intode}, \ref{smoothextcond}) that $\lim_{t\to 0}a(t)=1$. 

Conversely, suppose the function $a$ satisfies $\lim_{t\to 0}a(t)=1$ so that $a$ smoothly extends to a solution of \eqref{MasterODE} on the interval $(-\varepsilon,T)$ for some $\varepsilon \in (0,\infty]$. Choosing the initial condition $b(0)=0$ we obtain a solution to \eqref{intode} defined in some neighborhood of $0$ which satisfies $b^{\prime}(0)=1$. Thus, there exists an interval $(-R,R)$ about $t=0$ on which $b$ is an odd function of $r$.

We obtain from \eqref{curvform} and \eqref{MasterODE}
\begin{equation}\label{curvwitha}
K(t)=\lambda -\frac{2\mu}{a(t)}
\end{equation}
which, assuming $\lim_{t\to 0}a(t)=1$, converges to $\lambda -2\mu$ as $t\to 0$. Because of the rotational symmetry of the metric $g$, we have $\nabla K(0)=0$. 
\end{proof}

More generally, we have the following result relating the qualitative behavior of solutions of \eqref{MasterODE} to the geometry of the metrics \eqref{speccoord}.
\begin{lem}\label{ODEtoGeomProp}
 If $a>0$ is defined on $(T_0,T_1)$ with $0\leq T_0<T_1\leq \infty$ and solves \eqref{MasterODE} for some $\lambda,\mu\neq 0$,  then the following relationships hold between the behavior of $a$ and the geometry of the metric $g$ given by \eqref{speccoord}:
 \begin{enumerate}
  \item \label{ODEtoGeom1}$g$ has positive curvature if and only if  $a$ is strictly increasing and has negative curvature if and only if $a$ is strictly decreasing;
  \item \label{ODEtoGeom2}$g$ has bounded curvature if and only if $a$ is strictly bounded away from $0$;
  \item\label{ODEtoGeom3} $g$ is complete if and only if 
  \begin{enumerate}
      \item \label{ODEtoGeom3a}$0=T_0<T_1<\infty$, $a(0)=1$ and $\int_0^{T_1} \frac{a(t)}{t} dt =\infty$; or
      \item \label{ODEtoGeom3b}$0=T_0<T_1=\infty$, $a(0)=1$ and $a$ is strictly bounded away from $0$; or
      \item \label{ODEtoGeom3d}$0<T_0<T_1=\infty$,  $\int_{T_0}^{2 T_0} \frac{a(t)}{t} dt=\infty$ and $a$ is strictly bounded away from $0$.
 \end{enumerate}
  \item \label{ODEtoGeom4} if $T_1=\infty$ and $\lim_{t\to \infty} a(t)=a_\infty>0$, then $g$ is asymptotic to the end of a flat cone of angle $2 \pi a_\infty^{-1}$;
\end{enumerate}
\end{lem}
\begin{proof}
Item \eqref{ODEtoGeom1} is an immediate consequence of \eqref{curvform}.  The converse direction of Item \eqref{ODEtoGeom2} follows from \eqref{curvwitha}. On the other hand, if $a$ is not strictly bounded away from the zero, then, as \eqref{MasterODE} is an autonomous first order \textsc{ode}, $T_1=\infty$ and $\lim_{t\to \infty} a(t)=0$ so $K\to -\infty$. 
Item \eqref{ODEtoGeom3a}  is an immediate consequence of the form of the metric and Lemma \ref{smoothext}.  Likewise, Items \eqref{ODEtoGeom3b} and Items \eqref{ODEtoGeom3d} will follow from the form of the metric and Lemma \ref{smoothext} provided we show that if $\lim_{t\to \infty} a(t)=0$, then $\int_{T_0+1}^\infty \frac{a(t)}{t} dt<\infty$.  However, from \eqref{MasterODE} it is clear that if $\lim_{t\to \infty} a(t)=0$, then in all cases there is a constant $C_1>0$, depending on $\mu$ and $\lambda$, and a $T_*>T_0$, depending on $a$,  so for $t>T_*$,
 $$a'(t)\leq -C_1 a(t)^3.$$
Standard \textsc{ode} comparison then implies that for $t>T_*$, we have 
$$
a(t)\leq \frac{1}{\sqrt{C_1 t-C_2}}$$
for some $C_2$ which proves the claim.

If $\lim_{t\to \infty}a(t)=a_{\infty}$, then $a(t)=a_{\infty}+O(e^{-t})$ and so $\int_{2 T_0}^\infty \frac{a(t)}{t}=\infty$.  Hence, by \eqref{intode} $b$ is definte on $(R_0, \infty)$ and has the asymptotics
$$
b(r)=a_{\infty}^{-1} r +O(1).
$$
so the associated metric is asymptotic to the end of the claimed cone.
\end{proof}

\section{The classification}\label{classi}
We now analyze the solutions of \eqref{MasterODE} in order to prove Theorem \ref{ListofSolitonsThm}.  The constant solution $a\equiv 0$  separates the space of solutions of \eqref{MasterODE} 
into positive and negative solutions. For our purposes it sufficies to study positive solutions. Let $\mathcal{S}^+$ be the $3$-dimensional space of triples $(\mu,\gamma,a)$ where $\mu, \gamma \in \R_*$ -- $\gamma$ possibly infinite -- and $a$ is a smooth positive real-valued function defined on some interval $I$ which solves   
\begin{equation}\label{SlaveODE}
a^{\prime}(t)=4\mu a(t)^2\left(\frac{a(t)}{\gamma} -1\right).  
\end{equation}
There are three natural symmetries of $\mathcal{S}^+$. First, we let $\R^+$ act by
$$
\alpha \cdot (\mu,\gamma,a)=\left(\frac{\mu}{\alpha^2},\gamma,\alpha a\right)
$$
The metrics $g$ associated to $a$ and $g_\alpha$ associated to $\alpha a$ are locally isometric, but not globally isometric. 
Second, we let $\R^+$ act by
$$
(\mu,\gamma,a)\cdot \beta=\left(\frac{\mu}{\beta^4},\gamma\beta^2,a\circ s_{\beta}\right)
$$
where $s_{\beta}$ denotes the map $t \mapsto t\beta^{-2}$. This corresponds to scaling the metric $g$ associated to $a$ by $\beta^2$. By Lemma \ref{smoothext}, if $0\in I$ and $a(0)=1$, then this is also true after acting by $\beta$.
Third, is the non-geometric symmetry where $\R$ acts by translation in $t$, i.e.~$a(t) \mapsto a(t-\tau)$ for $\tau \in \R$. This action leaves $\mu,\gamma$ unchanged. 

There are six orbits of the actions corresponding to $(\mu,\gamma)\in \set{-{1},{1} }\times \set{-{1},{\infty}, {1}}$ and within each orbit there are one-parameter families of geometrically distinct solitons.  Due to the fact that certain solutions to \eqref{SlaveODE} blow-up, some orbits contain more than one such family.

\subsection{The case where $\gamma=\infty$}  
In this case $\lambda=0$ so we are considering metrics of steady solitons.
Moreover, \eqref{MasterODE} simplifies to 
$$
a^{\prime}(t)=-4\mu a(t)^2 
$$
whose solutions are given explicitly by 
$$
a(t)=\frac{1}{4\mu t+\varphi}
$$
for some constant $\varphi$. Qualitatively, this $2$-parameter family of solutions splits into three types, depicted in Figure \ref{FigLMZ}.\begin{figure} \label{FigLMZ}
\begin{center}
\def\svgwidth{5in}
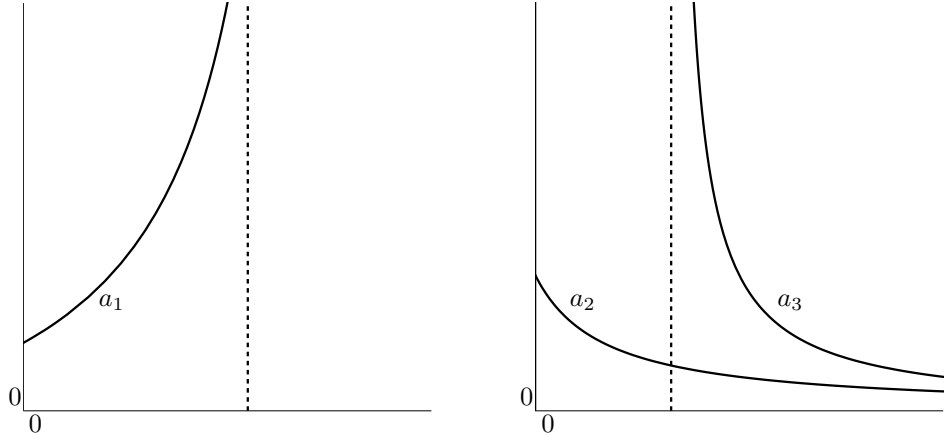
\caption{On the left solutions with $\gamma=\infty$ and $\mu<0$ on the right solutions with $\gamma=\infty$ and $\mu>0$.} 
\end{center}
\end{figure}
\subsubsection{$\mu=-\nu^2<0$} We may assume $\varphi>0$, otherwise $a<0$. It follows that  $a^{\prime}>0$ is strictly increasing and hence the solutions blow up in finite time. Using the smooth extension condition of Lemma \ref{smoothext} we obtain the solutions
$$
a_1(t)=\frac{1}{1-4\nu^2t}
$$
whose associated complete metrics $g_1(\nu)$ are Hamilton's steady \textit{cigar solitons}~\cite{MR954419} which are defined on the topological disk and given in polar coordinates by
$$
g_1(\nu)=\d r^2+\frac{1}{\nu^2}\tanh^2(\nu r)\d \theta^2. 
$$
By inspection these solutions are complete and have bounded positive curvature and are asymptotic to a cylinder of radius $\nu$.  Moreover, $g_1(\nu)=\nu^2 g_1(1)$.
\subsubsection{$\mu=\nu^2 >0$ and $\varphi>0$} The solutions exist for all (positive) times, are bounded, and approach $0$ as $t \to \infty$. It follows from Lemma \ref{ODEtoGeomProp} that the associated metrics have unbounded negative Gauss curvature and are incomplete. 
Using the smooth extension condition of Lemma \ref{smoothext} we obtain the solutions
$$
a_2(t)=\frac{1}{1+4\nu^2t}
$$
whose associated family of metrics -- called \textit{exploding solitons} in~\cite{MR2302600} -- $g_2(\nu)$ are given in polar coordinates by 
$$
g_2(\nu)=\d r^2+\frac{1}{\nu^2}\tan^2(\nu r)\d \theta^2. 
$$
Clearly, $g_2(\nu)=\nu^2 g_2(1)$.
Furthermore, at $r=\frac{\pi}{2\nu}$ the metrics have the asymptotics
$$
g_2(\nu)=\d r^2+\left(\frac{1}{\nu^4 \left( r-\frac{\pi}{2\nu}\right)^2}+O(1)\right)\d \theta^2. 
$$
\subsubsection{$\mu>0$ and $\varphi \leq 0$} The positive part of the solution exists on the time interval $(-\frac{\varphi}{4\mu},\infty)$, and approaches $\infty$ as $t \to -\frac{\varphi}{4\mu}$ and $0$ as $t \to \infty$.  By Lemma \ref{ODEtoGeomProp},  the associated metrics $g_3$ have unbounded negative Gauss curvature and are incomplete. 
Writing $\nu^2=- \varphi$, 
and applying the $\beta$ action with $\beta=\sqrt{\mu}$, gives 
$$
a_3(t)=\frac{1}{ 4 t- \nu^2}.
$$
In cylindrical coordinates $(\rho, \theta)$ this gives metrics $g_3(\nu)$ are given for $\nu>0$ as
\begin{equation*}
 g_3(\nu)=\d\rho^2 + \frac{ \nu^2}{\tanh^2 \left(  \nu \rho\right) } \d\theta^2
\end{equation*}
while for $\nu=0$
\begin{equation*}
 g_3(0)=\d\rho^2 +  \frac{1}{ \rho^2 } \d\theta^2.
\end{equation*}
Thus, the metrics are all asymptotically cylindrical with radius $ \nu$ as $\rho\to -\infty$. At $\rho=0$, the metrics are asymptotically
$$g_3(\nu)=\d\rho^2 +  \left(\frac{1}{ \rho^2 } +O(1)\right)\d\theta^2$$
and so to leading order behave like $g_2(1)$ as desired. 
\subsection{The case where $\infty>\gamma>0$} Besides the separatrix $a\equiv 0$, there is a positive separatrix $a\equiv \gamma$ and the orbits split into six families, depicted in Figure \ref{FigLMP}.

\begin{figure}\label{FigLMP}
\begin{center}
\def\svgwidth{5in}
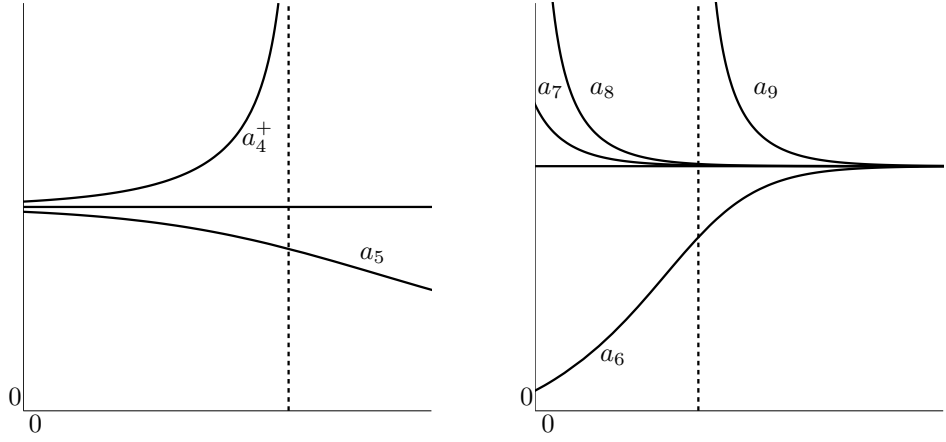
\caption{On the left solutions with $\infty>\gamma>0$ and $\mu>0$ on the right solutions with  $\infty>\gamma>0$  and $\mu<0$.} 
\end{center}
\end{figure}

\subsubsection{ $\mu>0$ and $a>\gamma>0$}\label{a4plus}
In this case the separatrix $a\equiv \gamma$ is unstable and so solutions starting above it are strictly increasing. If $1>\gamma$, then there is a solution $a_4^+$ which satisfies $a_4^+(0)=1$.  An \textsc{ode} comparison argument implies that the $a_4^+$ blow-up at some finite time $T_1(\mu,\gamma)>0$ given by
\begin{equation*}
 4\mu T_1(\mu,\gamma) = -1-\frac{1}{\gamma} \log (1-\gamma).
\end{equation*}
By applying an appropriate $\beta$ action, we may take $T_1=\frac{1}{4}$ and so $\mu=\mu(\gamma)$ and $\lambda=\lambda(\gamma)$.
The solutions have the following asymptotics at $T_1=\frac{1}{4}$:
\begin{equation*}
 a_4^+(t)=\frac{1}{\sqrt{\lambda(1-4t)}}+\frac{\gamma}{3}+o\left(1 \right).
\end{equation*}
Clearly,  $a_4^+$ is integrable and hence Lemma \ref{smoothext} and Lemma \ref{ODEtoGeomProp} imply that there are  positively curved, incomplete metrics, $g_4^+(\gamma)$, on the disk $\mathbb{D}$.   

One verifies that each $g_4^+(\gamma)$ extends smoothly to $\bar{\mathbb{D}}$ and, as $T_1=\frac{1}{4}$ the length of $\partial \mathbb{D}$ is $2\pi$.  Moreover, it follows directly from \eqref{geocurv} that $\partial \mathbb{D}$ is totally geodesic.  Finally, $g_4^+(\gamma)$ converges to the flat metric on the disk $\mathbb{D}(1)$ as $\gamma\to 0$, while $g_4^+(\gamma)$ converges to the metric $g_1(1)$ as $\gamma\to 1$.
Hence,  $\dist(0, \partial \mathbb{D})\to \infty$ as $\gamma\to 0$ while $\dist(0,\partial \mathbb{D})\to 1$ as $\gamma\to 1$ and $\dist(0, \partial \mathbb{D})$ is a strictly decreasing function of $\gamma$.  We take the parameter $\nu$ to be this distance.

\subsubsection{ $\mu>0$ and $a<\gamma$}
Again the separatrix $a\equiv \gamma$ is unstable and so  solutions  starting below it are decreasing and tend to $0$ as $t\to \infty$. If $1<\gamma$, then there is a solution $a_5$ which satisfies $a_5(0)=1$ and so, by Lemma \ref{smoothext} and Lemma \ref{ODEtoGeomProp}, there exist  negatively curved, incomplete metrics $g_5(\gamma,\mu)$ with unbounded curvature on $\Real^2$.

Furthermore, $a_5$ has the asymptotic form as $t\to \infty$
\begin{equation*}
 a_5(t)=\frac{1}{1+4\mu t} +\frac{1}{\gamma} \frac{\log (1+4\mu t)}{(1+4\mu t)^2}+o\left( \frac{\log (1+4\mu t)}{(1+4\mu t)^2}\right).
\end{equation*}
Hence, $a_5(t)$ is a asymptotic to $a_2(t)$ as $t\to\infty$ and so taking $\mu=\nu^2$, $g_5(\nu)$ is asymptotic to $g_2(\nu)$ as claimed.

\subsubsection{ $\mu<0$ and $a<\gamma$}
In this case the separatrix $a\equiv \gamma$ is stable and so solutions starting below it are increasing and  $a(t)\to \gamma$ as $t\to \infty$. 
If $1<\gamma$, then there is a solution $a_6$ which satisfies $a_6(0)=1$ and so, by Lemma \ref{smoothext} and Lemma \ref{ODEtoGeomProp}, there exists a family of positively curved, complete metrics $g_6(\nu)$ on $\Real^2$ which are asymptotic to a cone of angle $\nu={2\pi}{\gamma^{-1}}\in (0, 2\pi)$.   

\subsubsection{ $\mu<0$ and $a>\gamma>0$}
In this case the separatrix $a\equiv \gamma$ is stable and so solutions starting above it are decreasing and $a(t)\to \gamma$ as $t\to \infty$. 
An \textsc{ode} comparison argument implies that solutions blow-up at some finite initial time $T_0$.  The geometry of the associated metrics  is quite different depending on whether $T_0<0$, $T_0=0$ or $T_0>0$.

If $T_0<0$ and $1>\gamma$, then there is a solution $a_7$ which satisfies $a_7(0)=1$ and so, by Lemma \ref{smoothext} and Lemma \ref{ODEtoGeomProp}, there exists a family of negatively curved, complete metrics $g_7(\nu)$ on $\Real^2$ which are asymptotic to a cone of angle $\nu=2\pi \gamma^{-1}\in (2\pi,\infty)$.   

If $T_0=0$, then any solution $a_8$ has the following asymptotics near $T_0=0$
\begin{equation*}
 a_8(t)= \frac{1}{\sqrt{-4\lambda t}}+o\left(\frac{1}{\sqrt{t}}\right)
\end{equation*}
and so  $\frac{a_8(t)}{t}$ is not locally integrable.  By the $\beta$ symmetry we may take $\lambda=-1$. Hence, by Lemma \ref{ODEtoGeomProp}, there is a family of negatively curved, complete metrics $g_8(\nu)$ on a cylinder which are asymptotic at one end to a cone with angle $\nu=2\pi \gamma^{-1} \in (0,\infty)$.   
By \eqref{curvwitha}, $K\to -1$ as $t\to 0$, while the length of the curve $\set{t=t_0}$ goes to $0$ as $t_0\to 0$ and hence $g_8(\nu)$ is asymptotic at the other end to a hyperbolic cusp.

If $T_0>0$, then any solution $a_9$ has the following asymptotics near $T_0>0$
\begin{equation*}
 a_9(t)= \frac{1}{\sqrt{-4\lambda (t-T_0)}}+o\left(\frac{1}{\sqrt{t-T_0}}\right)
\end{equation*}
and so $\frac{a_9(t)}{t}$ is locally integrable.  Hence using Lemma \ref{ODEtoGeomProp}, it follows that there is a family of negatively curved, incomplete metrics on $\mathbb{D}_*$ which are asymptotic to a cone with angle $\nu=2\pi\gamma^{-1}$ at the puncture. As with the family $g_4(\nu)$, the metrics smoothly extend to $\bar{\mathbb{D}}_*$ and $\partial   \bar{\mathbb{D}}^*$ is totally geodesic.  By the $\beta$ symmetry we may take the length of $\partial \mathbb{D}_*$ to be $2\pi$ which gives the $g_9(\nu)$.

\subsection{The case where  $-\infty<\gamma<0$}  In this case there is no positive separatrix and there are four qualitatively different solutions as pictured in Figure \ref{FigLMN}.  
\begin{figure}\label{FigLMN}
\begin{center}
\def\svgwidth{5in}
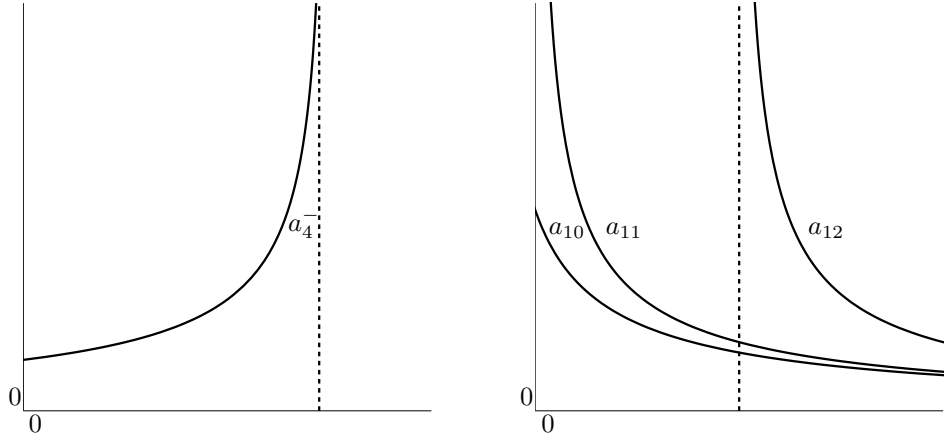
\caption{On the left solutions with $-\infty<\gamma<0$ and $\mu>0$ on the right solutions with  $-\infty<\gamma<0$ and $\mu<0$.} 
\end{center}
\end{figure} 

\subsubsection{$\mu<0$}
In this case the separatrix $a\equiv 0$ is unstable and so solutions starting above it are strictly increasing, moreover there exists a solution $a_4^-$ which satisfies $a_4^-(0)=1$. Qualitatively this solution is the same as the solution $a_4^+$ which satisfies $a_4^+(0)=1$. Indeed, by the $\beta$ symmetry, we may take the blow-up time to be $T_1=\frac{1}{4}$ and so $\mu=\mu(\gamma)$ and $\lambda=\lambda(\gamma)$. Hence, the metrics $g_4^-(\gamma)$ then have similar properties as the metrics produced in Section \eqref{a4plus}.  However, as $\gamma\to -\infty$ the metrics $g_4^-(\gamma)$ converge to the round metric on the hemi-sphere and so $\nu=\dist(0, \partial \mathbb{D})\in (\frac{\pi}{2}, \infty)$.

\subsubsection{$\mu>0$}
In this case the separatrix $a\equiv 0$ is stable and so solutions starting above it are decreasing and $a(t)\to 0$ as $t\to \infty$. 
An \textsc{ode} comparison argument implies that solutions blow-up at some finite initial time $T_0$.  The geometry of the associated metrics  is quite different depending on whether $T_0<0$, $T_0=0$ or $T_0>0$.

If $T_0<0$, then there is a solution $a_{10}$ which satisfies $a_{10}(0)=1$ and so, by Lemma \ref{smoothext} and Lemma \ref{ODEtoGeomProp}, there exists a family of negatively curved, incomplete metrics $g_{10}(\nu)$ with unbounded curvature on $\Real^2$ which are asymptotic to $g_2(\nu)$.

If $T_0=0$, then any solution $a_{11}$ has the following asymptotics near $T_0=0$
\begin{equation*}
 a_{11}(t)= \frac{1}{\sqrt{-4\lambda t}}+o\left(\frac{1}{\sqrt{t}}\right)
\end{equation*}
and so  $\frac{a_{11}(t)}{t}$ is not locally integrable. By the $\beta$ symmetry we may take $\lambda=-1$.  Hence, by Lemma \ref{ODEtoGeomProp}, there is a family of negatively curved, incomplete metrics with unbounded curvature on a cylinder. These metrics are asymptotic at one end to a hyperbolic cusp and asymptotic at the other end  to a member of the family $g_2(\nu)$ and give the family $g_{11}(\nu)$ as claimed. 

If $T_0>0$, then any solution $a_{12}$ has the following asymptotics near $T_0>0$
\begin{equation*}
 a_{12}(t)= \frac{1}{\sqrt{-4\lambda (t-T_0)}}+o\left(\frac{1}{\sqrt{t-T_0}}\right)
\end{equation*}
and so $\frac{a_{12}(t)}{t}$ is locally integrable.   Hence, by Lemma \ref{ODEtoGeomProp}, there is a family of negatively curved, incomplete metrics with unbounded curvature on $\mathbb{D}^*$. As with $g_4$, these metrics smoothly extend to $\bar{\mathbb{D}}_*$ and $\partial   \bar{\mathbb{D}}_*$ are totally geodesic. By the $\beta$-action $\partial_1 \mathbb{D}_*$ may be taken to have length $2\pi$. At the other end, these metrics are asymptotic to a member of the $g_2(\nu)$ family at the puncture and give the family $g_{12}(\nu)$ as claimed. 

This completes the analysis needed to prove Theorem \ref{ListofSolitonsThm}.

\subsection{Proof of Theorem 2} 
We now prove Theorem \ref{MainClassificationThm}.  First a simple lemma:
\begin{lem} \label{dumblem}
 Suppose $b\in C^\infty((0,R))$ is given by Lemma \ref{smoothext}. If $g=\d r^2 +b(r)^2 \d \theta^2$ is a non-constant curvature metric on $\mathbb{D}_*(R)$, then $Isom^+(\mathbb{D}_*(R),g)$, the group of orientation-preserving isometries, is isomorphic to $\mathbb{S}^1$.  
\end{lem}
\begin{proof}
Clearly, for each $\phi\in \mathbb{S}^1$, the maps $\psi_\phi :(r, \theta)\mapsto (r, \theta+\phi)$ are orientation preserving isometries.
Hence,  $\mathbb{S}^1$ may be identified with a subgroup of $Isom^+(\mathbb{D}_*(R),g)$. 
On the other hand, as solutions to \eqref{MasterODE} are either strictly monotone or constant curvature, \eqref{curvwitha} implies that if  $\psi:(\mathbb{D}_*(R), g)\to (\mathbb{D}_*(R), g)$  is an isometry, then $r(p)=r(\psi(p))$ and so $\psi=\psi_\phi$ for some $\phi$.  
\end{proof}

\begin{proof}[Proof of Theorem \ref{MainClassificationThm}]
 As $M_*$ is non-empty,  Proposition \ref{densityprop} implies that $K\neq 0$ and $M\backslash M_*$ consists of  isolated points $p_i\in M$ -- in particular, $M_*$ is connected.  
 
By the analysis of Section \ref{Sec2} and Theorem \ref{ListofSolitonsThm}, for each point $p\in M_*$ there is a simply-connected neighborhood $V_p\subset M_*$ about $p$ and an isometric embedding $\imath_p:V_p\to N_p$ where $(N_p,g_p)$ is a scaling of a member of one of the twelve families of by Theorem \ref{ListofSolitonsThm}.  
Clearly, for $q \in M_*$ with $V_p\cap V_q\neq \emptyset$, we have $(N_p,g_q)=(N_q,g_q)$ and there is an element $\psi_{p,q}\in Isom(N_p,g_p)$ so that $\psi_{p,q}\circ \imath_q=\imath_p$ on $V_p\cap V_q$. As $M_*$ is connected, we conclude that $(N_p,g_p)=(N,\beta^2 g)$ for $\beta>0$ and $(N,g)$ a fixed member of one of the families of Theorem \ref{ListofSolitonsThm}. As $p\in M_*$, we may delete a point -- if needed -- and take $(N,g)=\mathbb{D}_*(R)$ for $0<R\leq \infty$ and $g=\d r^2+b(r)^2 \d \theta^2$.

Given points $p,q \in M_*$, a continuous path $\gamma : [0,1]\to M_*$ connecting $p$ and $q$, and a fixed isometric embedding $\imath_p:V_p\to M$ there is a unique isometric embedding $\imath_{\gamma}:V_q\to N$ determined by $\imath_p$ and $\gamma$ in the obvious way -- this is a sort of developing map into $(N,g)$. For a fixed point $p_0\in M_*$, there is a group homomorphism $\Gamma: \pi_1(M_*, p_0)\to Isom(N,g)$ given by the holonomy of the above construction which is the obstruction to extending a choice of $\imath_{p_0}:V_{p_0}\to N$ to an isometric immersion $\imath:M_*\to N$. As $M$ is orientable as are (by inspection) all the models of Theorem \ref{ListofSolitonsThm} we have that $\Gamma: \pi_1(M_*, p_0)\to Isom^+(N,g)$.  Furthermore, as $\pi_1(M_*)$ is cyclic,  there is a generator $x\in \pi_1(M_*,p_0)$. We consider $\psi=\Gamma(x)$. By Lemma, \ref{dumblem} there is a $0<\phi\leq 2\pi$ so $\psi=\psi_\phi$ where $\psi_\phi:(r,\theta)\mapsto (r, \theta+\phi)$. Setting $\alpha=2\pi \phi^{-1} $ concludes the proof.
\end{proof}

\section{A variational characterization of two-dimensional gradient Ricci solitons}\label{VarSec}
Let $M$ denote a smooth two-manifold. On the space of Riemannian metrics $\mathcal{M}_{*}(M)$ of non-vanishing Gauss curvature on $M$ consider the functional defined by
$$
\mathcal{E}[g]=\int_M K_g\log |K_g|\mu_g
$$
where $\mu_g$ is the area-form of $g$. This functional has been applied to the study of Ricci flow on surfaces by Hamilton~\cite{MR954419} and Chow~\cite{Chow1991} -- in particular Hamilton observed that it is monotonically increasing along the Ricci flow on spheres with positive Gauss curvature.

In \cite{JacobThomasEntropy}, we observed that if $g$ is a positively curved metric which is critical for $\mathcal{E}$ with respect to compactly supported conformal variations, then $\hat{g}=K_g^{-3/2} g$ is the metric of a minimal surface in Euclidean $3$-space and that every negatively curved minimal surface metric arises in this way.  As such, it is interesting to consider critical points of $\mathcal{E}$ under other variations.  In this direction, we show that gradient Ricci soliton metrics are the only critical points of $\mathcal{E}$ with respect to area-preserving variations.

We say that $g \in \mathcal{M}_{*}(M)$ is an $\mathcal{E}$-critical metric if $\mathcal{E}$ is stationary at $g$ with respect to compactly supported area-preserving variations. Let $g$ be a smooth Riemannian metric and $h$ a smooth symmetric $2$-form on $M$. Writing $g_t=g+th$, we have
$$
\left.\frac{\partial}{\partial t}\right|_{t=0}\mu_{g_t}=\frac{1}{2}\,H\, \mu_g.
$$
where $H=\tr_g\! h$ is the trace of $h$ with respect to $g$. Hence, the variations that preserve area are precisely the deformations by trace-free symmetric $2$-forms. Moreover, we have (cf.~\cite[p.~99]{chowbook})
$$
\left.\frac{\partial}{\partial t}\right|_{t=0}K_{g_t}=- \frac{1}{4} \Delta_gH + \frac{1}{2} \Div_g\left(\Div_g\! \mathring{h}\right)- \frac{1}{2} HK_g 
$$
where $\mathring{h}$ denotes the trace-free part of $h$. A short computation yields
\begin{equation*}
 \delta_h\mathcal{E}[g]= \frac{1}{2}\int_{M}H K_g \log|K_g|+\left(\Div_g\left(\Div_g\! \mathring{h}\right)-\frac{1}{2} \Delta_g H - HK_g \right) (\log |K_{g}|+1)\vol.
\end{equation*}
If $h$ is compactly supported 
then integrating by parts gives
\begin{equation}\label{EuLag}
\delta_h\mathcal{E}[g]= -\frac{1}{4}\int_{M} H \left(\Delta_g \log |K_g| +2  K_g   \right)  -2\langle h, \mathring{\nabla}^2 \log |K_g|\rangle_g \vol,
\end{equation}
where $\langle a,b\rangle_g$ denotes the natural bilinear pairing on elements $a,b\in \Gamma(S^2(T^*M))$ obtained via $g$.
For $h$ trace-free \eqref{EuLag} simplifies to 
$$
\delta_h\mathcal{E}[g]=\frac{1}{2}\int_M\langle h, \mathring{\nabla}^2 \log |K_g|\rangle_g \,\vol.
$$
Hence, $g$ is an $\mathcal{E}$-critical metric if and only if the Gauss curvature $K$ of $g$ satisfies
$$
\mathring{\nabla}^2\log |K|=0. 
$$
The functional $\mathcal{E}$ is diffeomorphism invariant, i.e.~ for $\phi : M \to M$,  a diffeomorphism,
$$
\mathcal{E}[\phi^*g]=\mathcal{E}[g]
$$
for every $ g\in \mathcal{M}_*(M)$. Using Noether's principle we obtain a conservation law for $\mathcal{E}$: Let $X \in \Gamma(TM)$ be a compactly supported vector field and let $\phi_t$ denote its time $t$ flow. Then we have
$$
\phi_t^*g=g+tL_Xg+o(t). 
$$
As a consequence of the diffeomorphism invariance we obtain at an $\mathcal{E}$-critical metric
$$
\begin{aligned}
0=\delta_{L_Xg}\mathcal{E}[g]&=-\frac{1}{4}\int_M\tr_g(L_Xg)\left(\Delta_g \log |K_g| +2  K_g   \right)\mu_g,\\
&=-\frac{1}{2}\int_M\left(\Delta_g \log |K_g| +2  K_g\right)\Div_g\!X\,\mu_g
\end{aligned}
$$
which vanishes for every compactly supported vector field $X$ on $M$ if and only if 
$$
\Delta_g \log |K_g| +2  K_g
$$
is constant. Using Proposition \ref{charsol} we have thus shown:
\begin{thm}
Let $(M,g)$ be a Riemannian two-manifold with non-zero Gauss curvature. Then $(M,g)$ is a gradient Ricci soliton if and only if $g$ is a critical point of the functional $\mathcal{E}$ with respect to compactly supported area-preserving variations.  
\end{thm} 

\providecommand{\bysame}{\leavevmode\hbox to3em{\hrulefill}\thinspace}
\providecommand{\MR}{\relax\ifhmode\unskip\space\fi MR }
\providecommand{\MRhref}[2]{%
  \href{http://www.ams.org/mathscinet-getitem?mr=#1}{#2}
}
\providecommand{\href}[2]{#2}

\end{document}